\documentclass[10pt]{amsart}
\theoremstyle{plain}
\usepackage{amsmath, amsfonts, amsthm, amssymb,amscd,bbold, hyperref}
\usepackage{graphics}
\usepackage{color}
\usepackage{epsfig}
\usepackage[all]{xy}
\graphicspath{ {Figures/} }

\title[Braids]{An elementary fact about unlinked braid closures}

\author{J. Elisenda Grigsby}
\thanks{JEG was partially supported by NSF grant number DMS-0905848 and NSF CAREER award DMS-1151671.}
\address{Boston College; Department of Mathematics; 301 Carney Hall; Chestnut Hill, MA 02467}
\email{grigsbyj@bc.edu}

\author{Stephan M. Wehrli}
\thanks{SMW was partially supported by NSF grant number DMS-1111680.}
\address{Syracuse University; Mathematics Department; 215 Carnegie; Syracuse, NY 13244, USA}
\email{smwehrli@syr.edu}

\theoremstyle{plain}

\newtheorem{lemma}{Lemma}
\newtheorem{proposition}{Proposition}
\newtheorem{corollary}{Corollary}

\theoremstyle{definition}

\newtheorem{question}{Question}

\newcommand{\Szabo}{{Szab{\'o}} }

\newcommand{\Z}{\ensuremath{\mathbb{Z}}}

\newcommand{\F}{\ensuremath{\mathbb{F}}}

\newcommand{\Id}{\ensuremath{\mbox{\textbb{1}}}}

\newcommand{\cI}{\ensuremath{\mathcal{I}}}

\newcommand{\Wedge}{\ensuremath{\Lambda}}

\newcommand{\CKh}{\ensuremath{\mbox{CKh}}}
\newcommand{\Kh}{\ensuremath{\mbox{Kh}}}

\newcommand{\spincs}{\ensuremath{\mathfrak{s}}}

\begin{document}
\bibliographystyle{plain}
\begin{abstract} Let $n \in \Z^+$. We provide a short Khovanov homology proof of the following classical fact: if the closure of an $n$--strand braid $\sigma$ is the $n$--component unlink, then $\sigma$ is the trivial braid.
\end{abstract}
\maketitle

Let $\mathcal{B}_n$ denote the $n$--strand braid group, $\Id_n \in \mathcal{B}_n$ the $n$--strand trivial braid, and $U_n$ the $n$--component unlink in $S^3$. Denote by $\widehat{\sigma}$ the closure of $\sigma \in \mathcal{B}_n$, considered as a link in $S^3$. The following fact first appears in the literature in \cite[Thm. 4.1]{CochranTrivBraid}:

\begin{proposition} \label{prop:Braid} Let $\sigma \in B_n$. If $\widehat{\sigma} = U_n$, then $\sigma = \Id_n$.
\end{proposition}

%We shall see that Theorem \ref{prop:Braid} has two quite short proofs, neither requiring the braid foliation techniques of \cite[Thm. 1]{MR1030509}. The first proof (Section \ref{sec:Braidpi1}) uses classical techniques. The second (Section \ref{sec:BraidKh}) uses Khovanov homology. Although the first is accessible to a broader audience, we 

The purpose of this note is to provide a short Khovanov homology proof of Proposition \ref{prop:Braid}. Although the classical proof contained in \cite{CochranTrivBraid} is straightforward, we hope the Khovanov homology proof will also be of interest, since it suggests ways in which algebraic properties of Khovanov homology--in particular, its module structure--can give information about braid dynamics.

It may be of interest to the reader that there is a parallel story--going through the double-branched cover operation--involving minimal complexity fibered links in connected sums of copies of $S^1 \times S^2$. 

Explicitly, let $Y_n$ denote $\#^n (S^1 \times S^2)$. For $L$ a fibered link with fiber $F$, we will abuse terminology and refer to $\chi(F)$ as the {\em Euler characteristic of $L$}.

Define \[\mathcal{L}_n := \{\ell \in \Z^+\,\,\vline\,\, \ell \leq (n+1) \mbox{ and } \ell \equiv (n+1) \mod 2 \}.\] Note that for each $\ell \in  \mathcal{L}_n$, it is straightforward to construct a fibered link, ${\bf L}_\ell \subset Y_n$, of Euler characteristic $1-n$. See Figure \ref{fig:FiberedLinksS1S2}.  The monodromy of ${\bf L}_\ell$ is trivial, and the pair $(Y_n, \bf{L}_\ell)$ is well-defined up to diffeomorphism.

\begin{figure}
\begin{center}
\resizebox{4in}{!}{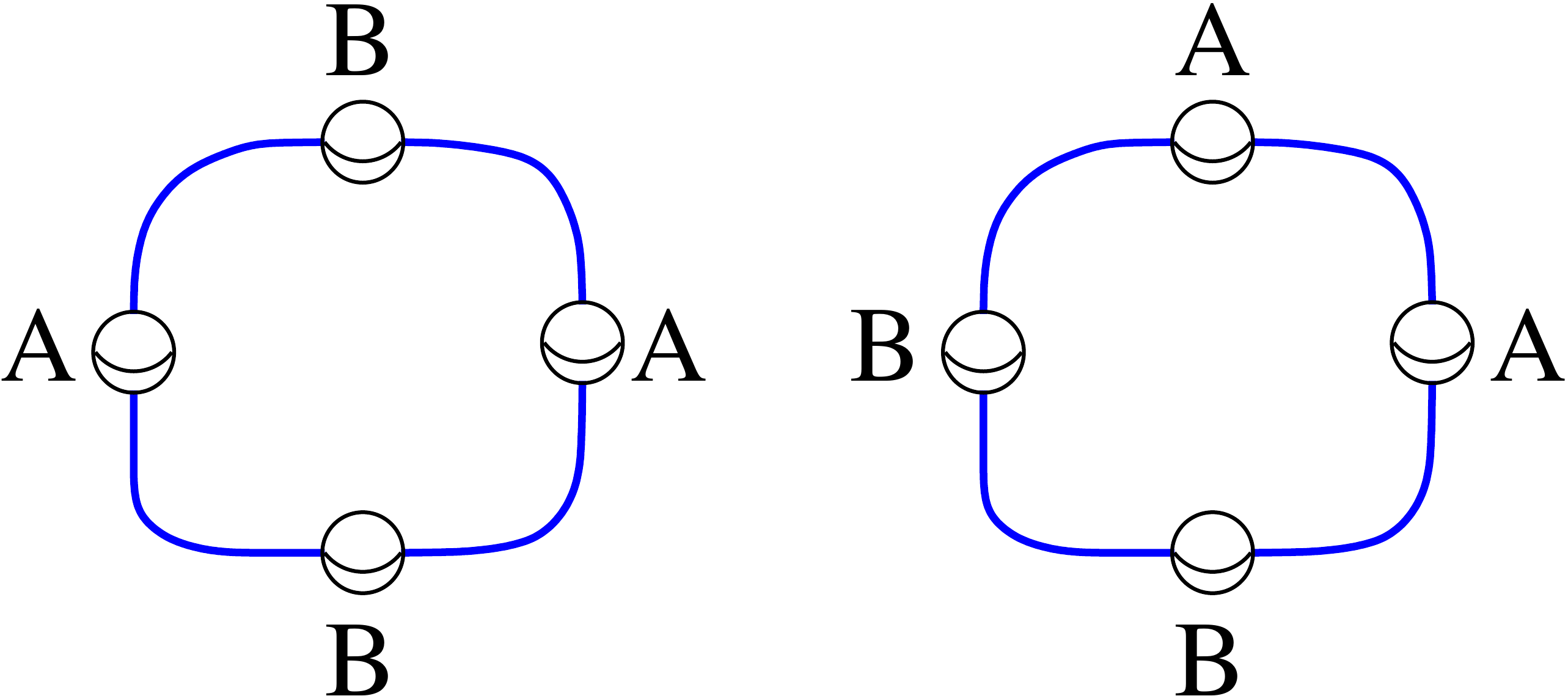}
\end{center}
\caption{Kirby diagrams of the links ${\bf L}_1$ (left) and ${\bf L}_3$ (right) in $Y_2 := \#^2 S^1 \times S^2$. The $S^2$'s (boundaries of the feet of $4$--dimensional $1$--handles) are identified as labeled, via a reflection in the plane perpendicular to the straight line joining their centers. The fibered link in each case is drawn in blue. To construct ${\bf L}_\ell \in Y_n$ in general, arrange $n$ pairs of $S^2$'s along an unknot in $S^3$ so that attaching $2$--dimensional one-handles to the disk bounded by the unknot, via the chosen configuration, forms an oriented surface with $\ell$ boundary components.}
\label{fig:FiberedLinksS1S2}
\end{figure}

The following result appears in \cite{NiHomAction}.  Indeed, after the first version of this note appeared, it was pointed out in \cite[Cor. 1.3]{GhigginiLisca} that Proposition \ref{prop:UniqueFiberedLink} implies Proposition \ref{prop:Braid}.

\begin{proposition} \label{prop:UniqueFiberedLink} \cite[Prf. of Thm. 1.3]{NiHomAction} Let $L_\ell \subset Y_n$ be a fibered, $\ell$--component link with $\ell \in \mathcal{L}_n$ and Euler characteristic $1-n$. Then the pair $(Y_n,L_\ell)$ is diffeomorphic to the pair $(Y_n, {\bf L}_\ell)$.
\end{proposition}

It is clear (cf. Lemma \ref{lem:ellinLn}) that if $\ell \not\in \mathcal{L}_n$, then an $\ell$--component link cannot have Euler characteristic $1-n$. It is also clear (cf. Lemma \ref{lem:maxchi}) that $1-n$ is the maximal possible Euler characteristic among all fibered links in $Y_n$. Informally, Proposition \ref{prop:UniqueFiberedLink} therefore says that for allowable $\ell$, maximal Euler characteristic fibered $\ell$--component links in $\#^n(S^1 \times S^2)$ are unique up to diffeomorphism.

In Section \ref{sec:UniqueFibLinkHF} we will give an alternative proof of Proposition \ref{prop:UniqueFiberedLink} that is formally analogous to the Khovanov homology proof of Proposition \ref{prop:Braid}.

We thank John Baldwin for pointing out that (this proof of) Proposition \ref{prop:UniqueFiberedLink} implies:

\begin{corollary} \label{prop:NoFiberedLinkHF} If $Y \not\cong Y_n$ is a closed, oriented $3$--manifold with the same Heegaard-Floer module structure as $Y_n$, then $Y$ contains no fibered links of Euler characteristic $1-n$.
\end{corollary}

There is a unique maximal Euler characteristic fibered link in $S^3$ (namely, the unknot) whose corresponding open book supports the standard tight contact structure. Ken Baker (cf. \cite{JesseJohnson}) asked the following interesting question:

\begin{question} Fix a contact structure, $\xi$, on a $3$--manifold, $Y$, and let \[\overline{\chi}_\xi := \mbox{max} \{\chi(L) \,\,|\,\, \mbox{$L$ is a fibered link whose open book supports $\xi$}\}.\] Up to diffeomorphism, are there finitely many fibered links $L$ supporting $\xi$ with $\chi(L) = \overline{\chi}_\xi$?
\end{question}

Proposition \ref{prop:UniqueFiberedLink} tells us that for the standard tight contact structure on $Y_n$ the answer is yes.
\vskip 10pt

\begin{flushleft}{\bf Acknowledgements:} We thank Ken Baker, John Baldwin, Rob Kirby, Tony Licata, and Danny Ruberman for interesting  conversations and Joan Birman and Bill Menasco for a useful e-mail correspondence. We are especially grateful to Ian Biringer for telling us about Hopfian groups, to Matt Hedden for pointing out that Proposition \ref{prop:UniqueFiberedLink} appears in \cite{NiHomAction}, and to Tim Cochran for making us aware that historical references to Proposition \ref{prop:Braid} in the literature appear under the slogan, ``Milnor's invariants detect the trivial braid."
\end{flushleft}

\section{Khovanov Homology Proof of Proposition \ref{prop:Braid}} 

\begin{proof}[Proof of Proposition \ref{prop:Braid}] Choose a diagram, $D(\widehat{\sigma})$, for $\widehat{\sigma}$ obtained as the closure of a diagram for $\sigma$, and mark the $n$ points on the diagram corresponding to the intersection with the closure arc. Recall that the ($\F = \Z/2\Z$) Khovanov homology, $\Kh(\widehat{\sigma})$, of $\widehat{\sigma}$ is an invariant of the isotopy class of $\widehat{\sigma} \subset S^3$ that takes the form of a bigraded vector space over $\F$. Since we have also chosen a basepoint on each of the $n$ link components, \cite[Prop. 1]{HedNiUnlink} tells us that $\Kh(\widehat{\sigma})$ inherits the structure of a module over the ring \[\mathcal{A}_n := \F[x_1, \ldots, x_n]/(x_1^2, \ldots, x_n^2)\] as follows.

Associated to the diagram of $\widehat{\sigma}$ is a cube of resolutions whose vertices are in one-to-one correspondence with complete resolutions (i.e., Kauffman states) of the diagram. The basis elements (generators) of the underlying vector space of the Khovanov chain complex, $\CKh(D(\widehat{\sigma}))$, are, in turn, in one-to-one correspondence with markings of the components of each resolution with either a $1$ or an $x$ (i.e., {\em enhanced} Kauffman states).

Let $\cI_{braid}$ be the unique ``braid-like" complete resolution of $D(\widehat{\sigma})$, and denote by $\Psi^+$ (resp., $\Psi^-$) the basis element $1 \otimes \ldots \otimes 1$ (resp., $x \otimes \ldots \otimes x$) in the vector space associated to $\cI_{braid}$. $\Psi^-$ is a cycle, hence represents an element in $\Kh(\widehat{\sigma})$. Indeed, $[\Psi^-] \in \Kh(\widehat{\sigma})$ is precisely Plamenevskaya's invariant  \cite{GT0412184} of the transverse isotopy class of the transverse link represented by $\widehat{\sigma}$. 

We are now ready to understand the $\mathcal{A}_n$ structure induced by the $n$ points $p_1, \ldots, p_n$. For each complete resolution, $\cI$,  choose a numbering of its $\ell_\cI$ connected components, and let $v_1 \otimes \ldots \otimes v_{\ell_\cI}$ represent the Khovanov generator whose $j$th component in $\cI$ is marked with $v_j \in \{1,x\}.$ Suppose $p_i$ lies on the $k$th component of $\cI$. Then the action of $x_i \in \mathcal{A}_n$ is the $\F$--linear extension of the assignment:
\[x_i \cdot (v_1 \otimes \ldots \otimes v_k \otimes \ldots \otimes v_{\ell_\cI}) := 
		v_1 \otimes \ldots \otimes x \otimes \ldots \otimes v_{\ell_\cI}\]
if $v_k = 1$ and $0$ otherwise.

It is straightforward to check that the Khovanov differential commutes with the action of $\mathcal{A}_n$, and it is shown in \cite{HedNiUnlink} (see also \cite{MR1740682}, \cite{MR2034399}) that the homotopy equivalences associated to Reidemeister moves respect the $\mathcal{A}_n$--module structure, and moving a basepoint past a crossing yields a homotopic map. The homology, $\Kh(\widehat{\sigma})$, therefore inherits the structure of an $\mathcal{A}_n$--module, and this $\mathcal{A}_n$--module structure is an invariant of the link.

With these preliminaries in place, assume that $\widehat{\sigma} = U_n$. A quick calculation using the standard diagram of $U_n$ tells us that $\Kh(U_n) \cong \mathcal{A}_n$ as an $\mathcal{A}_n$-module. Let $\theta \in \CKh(D(\widehat{\sigma}))$ be a cycle representing the homology class $1 \in \Kh(U_n) \cong \mathcal{A}_n$. 

We now claim that when $\theta$ is expressed as a linear combination of the standard Khovanov generators, the coefficient of $\Psi^+$ must be $1$. To see this, note that $x_1 \cdots x_n(\theta)$ represents the non-zero homology class $x_1 \cdots x_n \in \Kh(\widehat{\sigma})$, but if $v$ is any basis element not equal to $\Psi^+$, then $x_1 \cdots x_n(v) = 0$. We see this immediately for $v \neq \Psi^+ \in \cI_{braid}$, and any complete resolution $\cI \neq \cI_{braid}$ contains at least one connected component intersecting the closure arc more than once, hence containing at least two basepoints $p_i, p_j$, $i \neq j$. We conclude that any basis element $v$ associated to $\cI \neq \cI_{braid}$ satisfies $x_ix_j(v) = 0$, hence also satisfies $x_1 \cdots x_n(v) = 0$.

The arguments in the previous paragraph imply that $x_1 \cdots x_n(\theta) = x_1 \cdots x_n(\Psi^+) = \Psi^-$, so $[\Psi^-] = x_1 \cdots x_n \in Kh(\widehat{\sigma})$. In particular, $[\Psi^-] \neq 0$.

But \cite[Prop. 3.1]{BaldGrigSKhHFL} then implies that $\sigma$ is {\em right-veering}.

Repeat the argument above on $m(\sigma)$, the mirror of $\sigma$, to conclude that $\sigma$ is
also {\em left-veering}. Since the only braid which is both left- and right-veering is the identity braid (cf. \cite[Lem. 3.1]{BaldGrigSKhHFL}), $\sigma = \Id_n$, as desired.

\end{proof}

\section{Fibred links in $\#^n(S^1 \times S^2)$}
Recall that $\mathcal{L}_n := \{\ell \in \Z^+\,\,\vline\,\, \ell \leq (n+1) \mbox{ and } \ell \equiv (n+1) \mod 2 \}.$

\begin{lemma} \label{lem:ellinLn} If an $\ell$--component link $L$ has Euler characteristic $1-n$, then $\ell \in \mathcal{L}_n$.
\end{lemma}

\begin{proof} Let $S$ denote the fiber surface of $L$, $\chi(S)$ its Euler characteristic, and $g(S)$ its genus. Then $\chi(S) = 1-n = (2-2g(S)) - \ell.$ Since $g(S) \in \Z^{\geq 0}$, we obtain $\ell \equiv (n+1) \mod 2$ and $\ell \leq n+1$.
\end{proof}

\begin{lemma} \label{lem:maxchi} If $L \subset Y_n$ is a fibered link, then $\chi(L) \leq 1-n.$
\end{lemma}

\begin{proof} Suppose $L$ has $\ell$ components, and let $S$ denote the fiber surface of $L$, and $h$ its monodromy. $H_1(S)$ is free of rank $1-\chi(S) = 2g(S) + (\ell - 1)$. Viewing $Y_n - L$ as the mapping torus of $h$ (cf. Section 2.1), we obtain a corresponding presentation of $H_1(Y_n) \cong \Z^n$ with $1-\chi(L)$ generators, hence $1-\chi(L) \geq n$.
\end{proof}

\subsection{Heegaard-Floer homology proof of Proposition \ref{prop:UniqueFiberedLink}} \label{sec:UniqueFibLinkHF}

We begin with some background on Heegaard-Floer homology.

\subsubsection{Heegaard-Floer module}

Recall that in \cite{MR2113019},  Ozsv{\'a}th-Szab{\'o} associate to a closed, oriented $3$--manifold $Y$ a graded vector space (for simplicity we work over $\F = \Z/2\Z$), $\widehat{HF}(Y)$, which splits over Spin$^c(Y)$, the set of spin$^c$ structures on $Y$:
\[\widehat{HF}(Y) = \bigoplus_{\spincs \in \mbox{Spin}^c(Y)} \widehat{HF}(Y,\spincs)\] For appropriate choices of symplectic and almost complex structures, $\widehat{HF}(Y)$ is the Lagrangian Floer homology of a natural pair of Lagrangian tori, $\mathbb{T}_{\bf \alpha}$ and $\mathbb{T}_{\bf \beta}$, in the $g$--fold symmetric product of a pointed Heegaard surface, $(\Sigma, w)$, for $Y$.

$\widehat{HF}(Y)$ can be given the structure of a module over $\Wedge^*(H_1(Y;\F)),$ as described in \cite[Sec. 4.2.5]{MR2113019}. Explicitly, let \[(\Sigma, {\bf \alpha} = \{\alpha_1, \ldots, \alpha_g\}, {\bf \beta} = \{\beta_1, \ldots, \beta_g\}, z)\] be a pointed, genus $g$ Heegaard splitting of $Y$, and consider $\zeta \in H_1(Y;\F)$. Ozsv{\'a}th-\Szabo define an associated chain map, \[A_\zeta: \widehat{CF}(\Sigma,\alpha,\beta,z) \rightarrow \widehat{CF}(\Sigma,\alpha,\beta,z),\] on the Heegaard-Floer chain complex as follows (\cite[Rmk. 4.20]{MR2113019}). Let ${\bf x}, {\bf y} \in \mathbb{T}_\alpha \cap \mathbb{T}_\beta$ be generators of the chain complex. Recall that $\pi_2({\bf x},{\bf y})$ denotes the set of domains in $\Sigma$ representing topological Whitney disks connecting ${\bf x}$ to ${\bf y}$, in the sense of \cite[Sec. 2.4]{MR2113019}.  If $\phi \in \pi_2({\bf x}, {\bf y})$, we follow the notation in \cite[Sec. 2.1]{NiHomAction}, letting $\partial_\alpha\phi := (\partial \phi) \cap \mathbb{T}_\alpha$, regarded as a $1$--chain with boundary ${\bf y} - {\bf x}$. 

Choose an immersed curve, \[\gamma_\zeta \subset \Sigma - \{\alpha_i \cap \beta_j\}_{i,j \in \{1, \ldots, g\}},\] representing $\zeta \in H_1(Y;\F)$ and define \[a(\gamma_\zeta, \phi) := \# \widehat{\mathcal{M}}(\phi) (\gamma_\zeta \cdot \partial_\alpha\phi),\] where $\gamma_\zeta \cdot \partial_\alpha\phi$ is the algebraic intersection number of $\gamma_\zeta$ and $\partial_\alpha\phi$. Then the chain map associated to $\zeta$ is given by: 

\[A_\zeta({\bf x}) = \sum_{{\bf y} \in \mathbb{T}_\alpha \cap \mathbb{T}_\beta} \hskip 5pt \sum_{\left\{\phi \in \pi_2({\bf x}, {\bf y}) \,\, \vline \,\,\mu(\phi) = 1, n_{w}(\phi) = 0\right\}} a(\gamma_\zeta, \phi)\cdot {\bf y}.\] 

The map $A_\zeta$ is well-defined (independent of the choice of $\gamma$) up to chain homotopy (cf. \cite[Lem. 2.4]{NiHomAction}). 

\subsubsection{Heegaard-Floer contact invariant}

We now recall the definition of the Heegaard-Floer contact invariant \cite{MR2153455}, following the alternative construction given in \cite{MR2577470}. Let $\xi$ be a contact structure on a closed, connected, oriented $3$--manifold $Y$. Then Giroux tells us \cite{MR1957051} that there exists some fibered link $L$ whose corresponding open book supports $\xi$. One can then build a Heegaard diagram for $-Y$ ($Y$ with the opposite orientation) using
\begin{itemize}
	\item a choice of {\em basis}, $\{a_1, \ldots, a_n\}$, for a page $S$ (of Euler characteristic $1-n$)  of the open book \cite[Sec. 3.1]{MR2577470}, and 
	\item the data of the monodromy, $h$, of the open book.
\end{itemize}

Honda-Kazez-Mati{\'c} then identify a distinguished cycle in the corresponding chain complex, $\widehat{CF}(-Y)$, and prove both that the class it represents in $\widehat{HF}(-Y)$ is invariant of the choices used in its construction and that it agrees with the contact invariant defined in \cite{MR2153455}.

We will need the following property of the contact invariant, which follows immediately from \cite[Thm. 1.4]{MR2153455} and \cite[Thm. 1.1]{MR2318562}:

\begin{lemma} \label{lem:RV} If $L \subset Y$ is a fibered link whose monodromy, $h$, is not right-veering, then the Heegaard-Floer contact invariant associated to the contact structure supported by $L$ is $0$.
\end{lemma}

We now proceed to the proof.

\begin{proof}[Proof of Proposition \ref{prop:UniqueFiberedLink}] Let $L_\ell \subset Y_n$ be an $\ell$--component fibered link of Euler characteristic $1-n$. Construct a corresponding Heegaard diagram for $-Y_n$ as in \cite[Sec. 3]{MR2577470}.

The module structure on $\widehat{HF}(-Y_n)$ has been computed in \cite[Lem. 9.1]{MR2113019}. Explicitly, $\widehat{HF}(-Y_n) \cong \mathcal{A}_n$ as a module over \[\Wedge^*(H_1(-Y_n;\F)) \cong \mathcal{A}_n := \F[\zeta_1, \ldots, \zeta_n]/(\zeta_1^2, \ldots, \zeta_n^2).\]  In particular, $\zeta_1\cdots \zeta_n \neq 0 \in \widehat{HF}(-Y_n)$.

We can understand the module action explicitly in our setting as follows. All of our notation matches \cite{MR2577470}. Examine the Honda-Kazez-Mati{\'c} Heegaard diagram $\Sigma = S_{1/2} \cup -S_0$ associated to the fibered link, $L_\ell$, and look at the ``uninteresting" half, $S_{1/2} \subset \Sigma$, where the $n$--tuple of intersection points representing the contact class lives. Choose a compatible {\em dual basis}, $\{\gamma_1, \ldots, \gamma_n\}$, of simple closed curves on $S_{1/2}$ satisfying $|a_i \cap \gamma_j| = \delta_{ij}$. The set of homology classes, $\{[\gamma_1], \ldots, [\gamma_n]\}$, obtained by viewing the $\gamma_i$ as $1$--cycles in $-Y_n$, forms a basis for $H_1(-Y_n;\F)$. Hence, for each $i \in \{1, \ldots, n\}$, the corresponding map on homology induced by the chain map $A_{[\gamma_i]}$ can be identified with $\zeta_i \in \mathcal{A}_n$.

Let $\theta \in \widehat{CF}(-Y_n)$ be any cycle representing $1 \in \widehat{HF}(-Y_n)$. Since $\zeta_1 \cdots \zeta_n \neq 0 \in \widehat{HF}(-Y_n)$, we know that there exists at least one generator ${\bf y} \in \mathbb{T}_\alpha \cap \mathbb{T}_\beta$ satisfying \[\langle A_{[\gamma_1]} \cdots A_{[\gamma_n]} \cdot \theta, {\bf y}\rangle \equiv 1 \mod 2.\] 

Associated to such a generator ${\bf y}$ is an odd number of corresponding Maslov index $n$ domains in $\pi_2(\theta,{\bf y})$, each of which can be realized as the sum of $n$ of the Maslov index $1$ domains contributing to the chain maps $A_{[\gamma_1]}, \ldots, A_{[\gamma_n]}$. Consider the local multiplicity of such a Maslov index $n$ domain, $\psi$, in the $4$ regions adjacent to one of the constituent intersection points, $x_i$, of the distinguished cycle ${\bf x} = (x_1, \ldots, x_n)$ representing the contact class. We know (see Figure \ref{fig:HeegaardSurf}) that the local multiplicity of $\psi$ in the two regions adjacent to $x_i$ that contain the basepoint, $z_0$, must be $0$ and also that the local multiplicity in the region adjacent to the unique intersection point between $\gamma_i$ and $a_i$ must be nonzero (hence positive, since $\psi$ is a sum of domains representing holomorphic disks). Since the fourth region must have non-negative multiplicity, we conclude that $x_i$ must be a corner, of multiplicity at least one, in the boundary of $\psi$, implying that $x_i$ must be a constituent intersection point of the generator ${\bf y}$. 

\begin{figure}
\begin{center}
\resizebox{5in}{!}{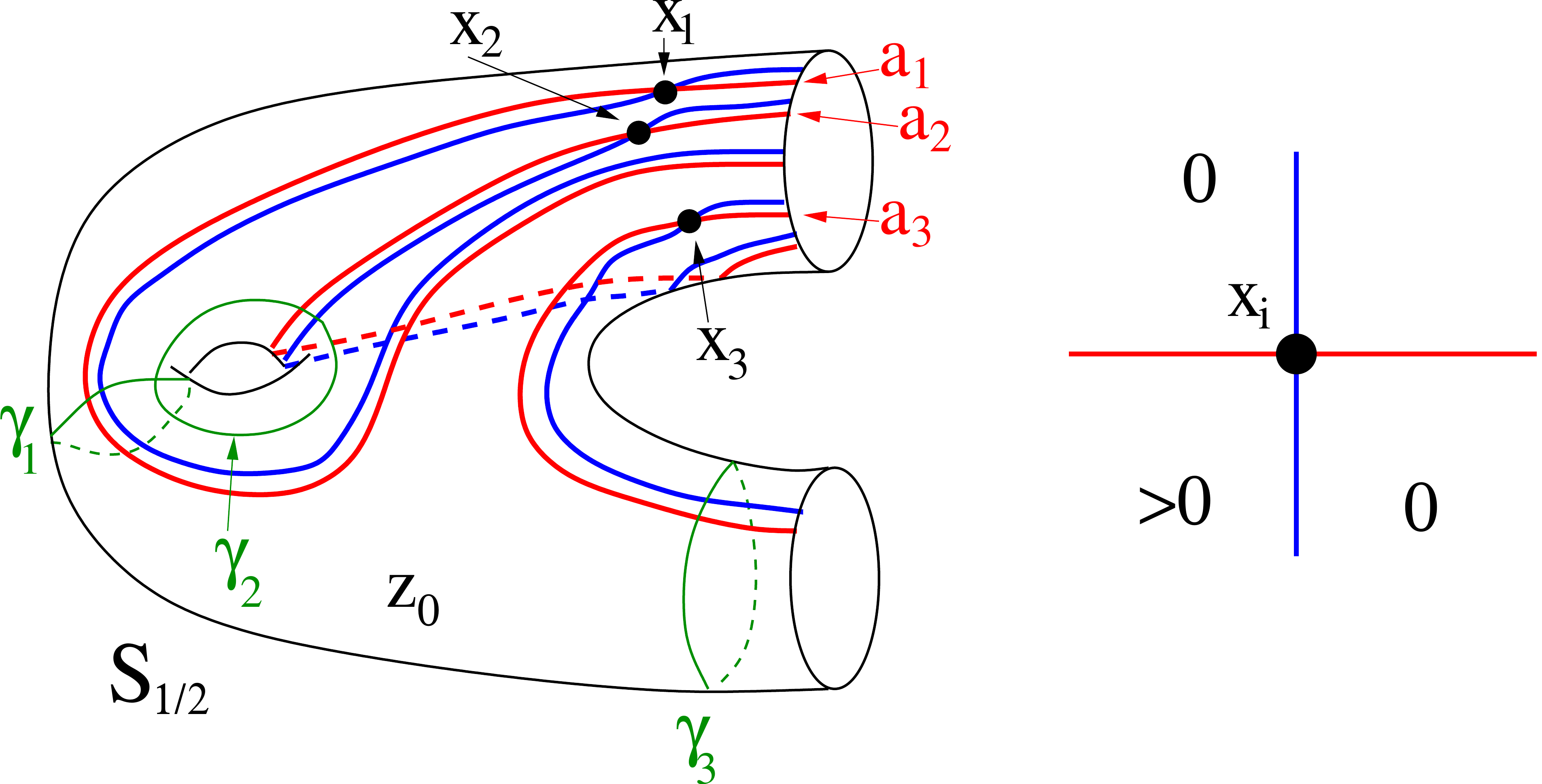}
\end{center}
\caption{The ``uninteresting" half of a Honda-Kazez-Mati{\'c} Heegaard diagram associated to a fibered link $L_2 \subset Y_3$. The right-hand picture is a close-up of one of the constituent intersection points of the contact class and restrictions on the local multiplicities of the Maslov index $n$ domain $\psi$. The NW, SE domains must have multiplicity $0$ since they contain the basepoint $z_0$. One of the other two domains must have positive multiplicity, since it is the unique domain intersecting $\gamma_i$.}
\label{fig:HeegaardSurf}
\end{figure}

Since the above argument holds for each of the $x_i$, we conclude that, in fact, ${\bf y}$ is actually the distinguished contact class, ${\bf x}$, and it follows that (working mod $2$) $A_{[\gamma_1]} \cdots A_{[\gamma_n]} \cdot \theta = {\bf x}$. Therefore, \[[A_{[\gamma_1]} \cdots A_{[\gamma_n]} \cdot \theta] = [{\bf x}] = \zeta_1 \cdots \zeta_n \neq 0 \in \widehat{HF}(-Y_n),\] so the Heegaard-Floer contact invariant associated to the contact structure supported by $L_\ell$ is nonzero. By Lemma \ref{lem:RV}, the monodromy, $h$, of $L_\ell$ is right-veering.

Now consider the {\em mirror} of $L$, i.e., the fibered link $L \subset -Y_n$ with monodromy $h^{-1}$. By running the same argument above, we conclude that the contact invariant associated to the contact structure supported by the mirror of $L$ is also nonzero. Hence, $h^{-1}$ is right-veering, implying that $h$ is left-veering.

But if $h$ is both right- and left-veering, then $h$ is isotopic to the identity mapping class, and hence $(Y_n, L_\ell)$ is diffeomorphic as a pair to $(Y_n, {\bf L}_\ell)$.
\end{proof}
\bibliography{BraidClosure_Unlink}
\end{document}